\documentclass[10pt, article]{amsart}
\usepackage{geometry} 
\usepackage{ae} % or {zefonts}
\usepackage[T1]{fontenc}
\usepackage[cp1250]{inputenc}
\usepackage{amsmath}
\usepackage{amssymb, amsfonts,amscd,verbatim}
\usepackage{mathtools}
\usepackage{MnSymbol}
\usepackage{mathrsfs} % for \mathscr

\usepackage[normalem]{ulem}
\usepackage{hyperref}
\usepackage{indentfirst}
\usepackage{latexsym}
\input xy
\xyoption{all}

\usepackage{amsmath}    % need for subequations
\theoremstyle{plain}
\newtheorem{Pocz}{Poczatek}[section]
\newtheorem{Proposition}[Pocz]{Proposition}

\newtheorem{Theorem}[Pocz]{Theorem}
\newtheorem{Corollary}[Pocz]{Corollary}

\newtheorem{Lemma}[Pocz]{Lemma}

\newtheorem{Question}[Pocz]{Question}
\newtheorem{Problem}[Pocz]{Problem}
\newtheorem{Example}[Pocz]{Example}

\theoremstyle{definition}
\newtheorem{Definition}[Pocz]{Definition}

\theoremstyle{remark}
\newtheorem{Remark}[Pocz]{Remark}

\def\RR{{\mathbb R}}

\def\ZZ{{\mathbb Z}}

\numberwithin{equation}{section}

\author{Kyle ~ Austin}
\address{University of Tennessee, Knoxville, USA}
\email{kaustin9@vols.utk.edu}
\author{Jerzy~Dydak}
\address{University of Tennessee, Knoxville, USA}
\email{jdydak@utk.edu}
\author{Michael ~ Holloway}
\address{University of Tennessee, Knoxville, USA}
\email{holloway@math.utk.edu}

\title[Scale Structures and C*-algebras
]%
  {Scale Structures and C*-algebras}

\date{ \today
}


\begin{document}
%\fontsize{18}{20pt}\selectfont

\maketitle
%\begin{center}
%\today
%\end{center}

\begin{abstract}
The purpose of this paper is to investigate the duality between large scale and small scale. 
It is done by creating a connection between C*-algebras and scale structures.
In the commutative case we consider C*-subalgebras of $C^b(X)$, the C*-algebra of bounded complex-valued functions on $X$. Namely, each C*-subalgebra $\mathscr{C}$ of $C^b(X)$ induces both a small scale structure on $X$ and a large scale structure on $X$. The small scale structure induced on $X$ corresponds (or is analogous) to the restriction of $C^b(h(X))$ to $X$, where $h(X)$ is the Higson compactification.
The large scale structure induced on $X$ is a generalization of the $C_0$-coarse structure of N.Wright. Conversely, each small scale structure on $X$ induces  a C*-subalgebra of $C^b(X)$ and each large scale structure on $X$ induces  a C*-subalgebra of $C^b(X)$. To accomplish the full correspondence
between scale structures on $X$ and C*-subalgebras of $C^b(X)$ we need to enhance the scale structures to what we call hybrid structures. In the noncommutative case we consider C*-subalgebras
of bounded operators $B(l_2(X))$.
\end{abstract}

\section{Introduction}
The purpose of this paper is to investigate the duality between large scale and small scale. Historically, large scale was introduced under the name of coarse structures in \cite{Roe} and what we mean by small scale was known under the name of uniform structures (see \cite{I} or \cite{J}).
For the purpose of exposing the duality we are unifying the terminology here.

A study of formal analogy/duality between the scales was initiated in \cite{CDV2}. Here we go much further in the form of creating functors between large scale and small scale. Some precursors of those functors were hidden in the following concepts/constructions:\\
1. Higson compactification of a proper coarse space (see \cite{Roe}),\\
2. $C_0$-coarse structures of N.Wright (see \cite{W1} and \cite{W2}),\\
3. $C_0$-coarse structures on uniform spaces (see \cite{MYY} and \cite{MY}),\\
4. Continuously controlled coarse structures (see \cite{Roe}).

It turns out that one way to express the duality of large scale and small scale structures on a set $X$ is via C*-subalgebras of $C^b(X)$, the C*-algebra of bounded complex-valued functions on $X$ (see \cite{ADH2} for a description of duality between the scales bypassing C*-subalgebras). Namely, each C*-subalgebra $\mathscr{C}$ of $C^b(X)$ induces both a small scale structure on $X$ and a large scale structure on $X$. The small scale structure induced on $X$ corresponds to the restriction of $C^b(h(X))$ to $X$, where $h(X)$ is the Higson compactification.
The large scale structure induced on $X$ is a generalization of $C_0$-coarse structures of N.Wright (see \cite{W1} and \cite{W1}). Conversely, each small scale structure on $X$ induces  a C*-subalgebra of $C^b(X)$ and each large scale on $X$ induces  a C*-subalgebra of $C^b(X)$.

As a byproduct of our approach we get a more elegant formulation of the classic Stone-Weierstrass Theorem (see Theorem \ref{GenStoneWeierstrass}).

In the noncommutative case we consider C*-subalgebras of bounded operators $B(l_2(X))$ and we show how they induce pure scale structures on $X$.
\section{Scale Structures}

\subsection{An Introduction to Scales}

One of the goals of this paper is to unify approaches to coarse structures and uniform structures. In coarse geometry, the basic concept is that of a uniformly bounded cover. In uniform topology, the basic concept is that of a uniform cover. We generalize those by thinking of covers as scales where the comparison is not via refinement or coarsening but by using stars of covers.

For a small scale or large scale structure, we need not only scales but also a way to compare scales. For example, in the small scale, one should be able to \lq zoom in\rq\ to the space by going to a smaller scale, and in the large scale one should be able to \lq zoom out\rq\ to a larger scale. To capture this idea, we use the relation of star refinement. 

\begin{Definition}
A \textbf{scale} on a set $X$ is a cover $\mathcal{U}$ of $X$. A scale $\mathcal{U}$ is \textbf{smaller}
than a scale $\mathcal{V}$ if $st(\mathcal{U},\mathcal{U})$ refines $\mathcal{V}$ and $\mathcal{U}\ne \mathcal{V}$. Obviously, in that case $\mathcal{V}$ is \textbf{larger} than $\mathcal{U}$.
Recall that $st(\mathcal{U},\mathcal{U})$ consists of stars $st(U,\mathcal{U})$, $U\in \mathcal{U}$, which are defined as the union of $V\in \mathcal{U}$ intersecting $U$.
 \end{Definition}
 
  A \textbf{scale structure} on $X$ is a collection of scales of $X$ with some order properties.

 Every scale structure on a space can be made into a partially ordered set by using the relation of star refinement. That is, $\mathcal{U} \leq \mathcal{V}$ if and only if $st(\mathcal{U}, \mathcal{U}) \prec \mathcal{V}$ (i.e., $st(\mathcal{U}, \mathcal{U})$ refines $\mathcal{V}$). 

A small scale structure will be a scale structure in which one can always zoom into the space and view it from smaller and smaller scales. Using the language of filters, a \textbf{small scale structure} is a filter of the set of all covers of a space, ordered by star refinement, and a base for a small scale structure is a filter base of this poset. Dually, a large scale structure on a space is a scale structure which allows one to zoom out and view the space from farther and farther away. That is, a  \textbf{large scale structure} is a filter of the set of all covers of a space, ordered by reverse star refinement, and a base for a large scale structure is a filter base of this poset. Another way to define these is to say that a small scale structure is a scale structure in which one can always decrease scale (unless there is a scale $\mathcal{U}$
such that $\mathcal{U}=st(\mathcal{U}, \mathcal{U})$ and $\mathcal{U}$ refines all other scales) and a large scale structure is one in which one can always increase scale (unless there is a scale $\mathcal{U}$
such that $\mathcal{U}=st(\mathcal{U}, \mathcal{U})$ and $\mathcal{U}$ coarsenes all other scales). In the remainder of this section, we shall first introduce the most important examples of scale structures, which are metric structures and translation structures, and then we will introduce the general definitions of small and large scale structures.

\subsection{Small Scale Structures}

The general theme is that in small scale geometry, one zooms inward on a space through  star refining covers, while in large scale geometry, one works outward on a space through  star coarsening covers. In order to exhibit the duality between large and small scales we will adjust the terminology accordingly.

Let's translate basic definitions from the theory of uniform spaces (see \cite{I} or \cite{J}) into the language of scales.  

\begin{Definition} \label{UniformDef}
A \textbf{small scale structure} (ss-structure for short) on a set $X$ is a filter  $\mathcal{SS}$ of scales on $X$. In other words:

1)  if $\mathcal{U}_1 \in \mathcal{SS}$ and $\mathcal{U}_2$ is a cover of $X$ which coarsens $\mathcal{U}_1,$ then $\mathcal{U}_2 \in \mathcal{SS}$;

2) if $\mathcal{U}_1$, $\mathcal{U}_2 \in \mathcal{SS}$ then there exists $\mathcal{U}_3 \in \mathcal{SS}$ which star refines both $\mathcal{U}_1$ and $\mathcal{U}_2$.

For simplicity, elements of $\mathcal{SS}$ will be called $\mathcal{SS}$-scales or s-scales if $\mathcal{SS}$ is clearly understood in a particular context.

A set equipped with an ss-structure is called a \textbf{small scale space} (or ss-space for short). 
\end{Definition}

\begin{Remark}
s-scales
were traditionally called \textbf{uniform covers} (see \cite{I} or \cite{J}).
\end{Remark}

\begin{Definition}
 A function $f: X \to Y$ of ss-spaces is \textbf{small scale continuous} (ss-continuous for short) if the inverse image of an s-scale of $Y$ is an s-scale of $X$.
Equivalently, for any s-scale $\mathcal{V}$ of $Y$ there is an s-scale $\mathcal{U}$ of $X$ such that $\{f(U) : U \in \mathcal{U}\} $ refines $\mathcal{V}.$
 \end{Definition}
 
 We say that a small scale structure $\mathcal{SS}$ is \textbf{Hausdorff} if for each $x\ne y\in X$ there exists a scale $\mathcal{U} \in \mathcal{SS}$ which has no set containing both $x$ and $y$.
 
 \begin{Remark}
 Given any family of scale structures we can consider their union or their intersection.
 That quickly leads to the concepts of smallest or largest scale structures satisfying certain conditions.
\end{Remark}

\subsubsection{Topology and Small Scale Structures}
Every small scale structure $\mathcal{SS}$ induces a \textbf{topology} on its underlying set as follows: a subset $U \subseteq X$ is \textbf{open} provided for each $x \in U$ there is a scale $\mathcal{V} \in \mathcal{SS}$ such that $st(\{x\},\mathcal{V}) \subset U$.

Notice that for each s-scale $\mathcal{U}$, its interiors form an s-scale (given $U\in \mathcal{U}$ and given an s-scale $\mathcal{V}$, the set of points $x$ such that
$st(st(x,\mathcal{V}),\mathcal{V})\subset U$ is open). Let's call  scales consisting of open sets \textbf{open scales}. That means open s-scales form a basis of any ss-structure.

\begin{Proposition}\label{PUs}
Let $(X,\mathscr{SS})$ be a small scale space inducing the topology $\mathscr{T}$. Any open s-scale $\mathcal{U}$ of $X$ (with respect to $\mathscr{T}$) has a continuous partition of unity subordinated to it.
\end{Proposition}
\begin{proof}
Given an open s-scale $\mathcal{U}$, one can create a decreasing sequence of open s-scales $\mathcal{U}_n$, $n\ge 1$, so that $\mathcal{U}_1=\mathcal{U}$. There is a simple proof in  \cite{AD} (see Theorem 3.4) that in such a case there is a partition of unity subordinated to $\mathcal{U}$, i.e. a family of non-negative continuous functions $\{\phi_U\}_{U\in \mathcal{U}}$ adding up to $1$ with the support of each $\phi_U$ contained in $U$. That is exactly what we mean by a partition of unity \textbf{subordinated} to $\mathcal{U}$.
\end{proof}

\begin{Corollary}
If the ss-structure is Hausdorff, then the induced topology is completely regular.
\end{Corollary}

It is less known that many concepts topologists have long been using, such as compactness, barycentric subdivision, and paracompactness, have origins in the uniform category. The subsequent results will make this more precise.

\begin{Proposition}\label{ssAndPUs}
Let $(X,\mathscr{T})$ be a topological space. The following are small scale structures on $X$:\\
1. The family of all covers having a finite continuous partition of unity subordinated to them.\\
2. The family of all covers having a continuous partition of unity subordinated to them.\\
If the topology $\mathscr{T}$ is completely regular, then both the above small scale structures induce $\mathscr{T}$, the first one is the smallest ss-structure inducing $\mathscr{T}$ and the second one is the largest ss-structure inducing $\mathscr{T}$.
\end{Proposition}
\begin{proof}
As described in \cite{DyPU}, each partition of unity $\phi$ has its derivative being subordinated to a cover that is a star-refinement (in the pointwise sense) of the cover of supports $\mathcal{U}$ of $\phi$. Therefore the second derivative of $\phi$ is subordinated to a cover that star-refines $\mathcal{U}$.

Notice that $\mathscr{T}$ being completely regular means exactly that for any open set $U$ and any $x\in U$ there is a finite partition of unity on $X$
containing a function $f$ with support in $U$ and $f(x)\ne 0$. That proves the last part of the proposition.
\end{proof}

\begin{Corollary}\label{para}
A topological Hausdorff space $X$ is paracompact if and only if the collection of open covers of $X$ forms a base for an ss-structure on $X$, and that ss-structure generates the original topology on $X$.
\end{Corollary}

 Traditionally, a topological space is said to be \textbf{uniformizable} if there exists a small scale structure on the space which induces the topology of the space. Thus, it must be completely regular. Conversely, if $X$ is a completely regular topological space, then \ref{ssAndPUs} says there is a small scale structure $\mathscr{SS}$ on $X$ which generates its topology. Note that $\mathscr{SS}$ is contained in every ss-structure in which all continuous functions $f: X \to \RR$ are ss-continuous. 
 
 \ref{ssAndPUs} implies that a compact Hausdorff space has a unique ss-structure inducing its topology. For metric spaces, having a unique small scale structure is equivalent to being compact, but in general, having a unique ss-structure only implies that the space is locally compact; see \cite{Do}. The long line
is a good example of a completely regular space with unique ss-structure inducing its topology since it has only one compactification.

Using \ref{ssAndPUs} one can see that, for a given topological space $X$, the small scale structures on $X$ inducing its topology are in one to one correspondence with compactifications $h(X)$ of $X$. In one direction there is the smallest compactification (called \textbf{Samuel compactification} - see \cite{J}) over which all ss-continuous functions from $X$ to $[0,1]$ extend continuously. In the other direction the compactification $h(X)$ has a unique ss-structure inducing its topology and that ss-structure is restricted to $X$.

\subsection{Large Scale Structures}

The following definition of a large scale structure is a dualization of the definition of a small scale structure on a set.

\begin{Definition} \cite{DH}
A \textbf{large scale structure} on a set $X$ (ls-structure for short) is a nonempty filter of scales on $X$
(in the order being the reverse star refinement) along with all refinements of those scales.

Alternatively, it is a collection $\mathscr{LS}$ of covers of subsets of $X$ satisfying the following property: if $\mathcal{B}_1$, $\mathcal{B}_2 \in \mathscr{LS}$ then $st(\mathcal{B}_1,\mathcal{B}_2) \in \mathscr{LS}$. 

If a family $\mathcal{B_1}$ can be coarsened to $\mathcal{B_2} \in \mathscr{LS}$, then we say that $\mathcal{B_1}$ is \textbf{uniformly bounded}.
Notice that $\mathcal{B}_1 \in \mathscr{LS}$ implies $\mathcal{B}_2$ is uniformly bounded if each nonsingleton element of $\mathcal{B}_2$ is contained in some element of $\mathcal{B}_1$.
\end{Definition}

%The reason the two ways of defining large scale structures are equivalent is as follows: 
%Let $(X,\mathcal{LS})$ be a large scale structure defined via uniformly bounded families, i.e. not necessarily scales.
Given a uniformly bounded family $\mathcal{B}$ we define the \textbf{trivial extension} of $\mathcal{B}$ to be $\mathcal{B} \cup \{\{x\}:x\in X\}$. The trivial extension of any uniformly bounded family is an l-scale, so although a uniformly bounded family may not be a cover of the underlying space, any such family may be extended to be a scale.
%By trivially extending all elements of $\mathcal{LS}$, we simply refer to its elements as uniformly bounded covers instead of collections. This is the equivalent of the diagonal being a controlled set in the setting of coarse structures.

A map $f: X \to Y$ between spaces with large scale structures  is called \textbf{large scale continuous} (or \textbf{bornologous}) if the image of any uniformly bounded family in $X$ is uniformly bounded in $Y$.

\begin{Remark}
 In analogy to topological spaces, i.e. sets with a topology, we can talk about 
 \textbf{small scale spaces} (ss-spaces) or \textbf{large scale spaces} (ls-spaces), i.e. sets with structures attached to them.
\end{Remark}

\subsection{Metric Scale Structures}

Each pseudometric space $(X, p)$ induces a small scale structure (called the \textbf{metric small scale structure}) by taking as scales the collections $\mathcal{B}_r = \{B(x,r):x\in X\}$ for $r>0$. This structure is perhaps the most important scale structure. A metric small scale structure includes all $\mathscr{B}_r$ along with any cover which coarsens a $\mathcal{B}_r.$ Thus, this structure contains all covers with positive Lebesgue number
(see \ref{Lebesgue number of a scale}). Notice that these covers can become arbitrarily large, so long as they have large overlap. From the small scale point of view, all that is important is that the cover has some thickness, not that it is bounded.

Notice that the metric small scale structure is Hausdorff if and only if $p$ is a metric.

A metric large scale structure consists of all $\mathcal{B}_r$ along with all covers which refine a $\mathcal{B}_r$. The idea is that anything which is smaller than a uniformly bounded collection is uniformly bounded. Notice that these collections can become arbitrarily thin, because from the large scale point of view, the only thing important is that there is a bound on the mesh of the cover.

In a metric space, the distance function gives a natural way to measure the scale of a cover. Namely, one can say how \lq thick\rq\ a cover is by using the Lebesgue number. We will redefine the Lebesgue number of a cover using the concept of larger/smaller scale. 
\begin{Definition}\label{Lebesgue number of a scale}
The \textbf{Lebesgue number of a scale} $\mathcal{U}$ is the supremum of the set of real numbers $\lambda$ such that the collection of $\lambda$-balls of the space is a smaller scale than $\mathcal{U}$. 
\end{Definition}
From this point of view, the Lebesgue number lemma says that all open covers of a compact metric space have positive thickness. Another meaning of it is that
all metrics inducing the same compact topology are small scale equivalent.

The Lebesgue number is a small scale concept. Dual to the Lebesgue number is the large scale notion of the \textbf{mesh} of a cover. 
\begin{Definition}
The \textbf{mesh of a scale} $\mathcal{U}$ is the infimum of the real numbers $M$ such that  $\mathcal{U}$ is smaller than the scale of $M$-balls. 
\end{Definition}

\begin{Remark}
 Notice our definitions of the Lebesgue number and of mesh differ from the traditional ones,
 as they are designed to reflect the order on scales. However,
 they are bi-Lipschitz equivalent to them which is all that matters.
\end{Remark}
Notice that mesh is exactly dual to the notion of Lebesgue number. The Lebesgue number and the mesh of a cover both quantify the scale of that cover, and both are determined by comparing covers using star refinement. Small scale and large scale structures on a space give a way to extend this notion of using star refinements to determine comparison of scales to a more general class of spaces.

\subsection{Translation Structures}\label{translation}

Let $G$ be a group. There are two well known ways of creating invariant scale structures on $G$, both of which are induced by translations. The first corresponds to the uniform (small scale) structure on $G$, and the other is a generalization of the coarse (large scale) structure induced by the Cayley graph of $G$.

 If $G$ is a topological group then there exists a neighborhood base of the identity element, $\{U_{\alpha}:\alpha \in A\}$, where $A$ is directed by square inclusion; i.e., $\alpha \ge \beta$ if and only if $U_{\alpha} \cdot U_{\alpha} \subset U_{\beta}$. This allows us to create a small scale structure on $G$ by declaring the base scales of $G$ to be $\mathcal{U}_{\alpha} = \{gU_{\alpha}:g\in G\}$. 
 
 If $G$ is a locally compact topological group, then considering scales of $G$ of the form $\mathcal{U}_{\alpha} = \{gU_{\alpha}:g\in G\}$, where $U_\alpha$ is a pre-compact  (i.e. $cl(U_\alpha)$ is compact) neighborhood of $1_G$ yields a large scale structure on $G$. 
If $G$ is any group, then we can put the discrete topology on it and the above large scale structure on $G$ has scales given by the collections $\mathcal{U}_F = \{ gF: g\in G\}$ where $F$ ranges over the finite subsets of $G$. It is known, see \cite{DH} and \cite{BDM}, that this scale structure is large scale equivalent to the metric scale structure induced by the Cayley graph in the case when $G$ is a finitely generated group. The advantage of this approach, however, is that one does not need to define a Cayley graph or restrict to countable groups in order to use proper metrics.

It is notable that the concept of a translation enables one to extend the methods of geometric group theory to metric spaces. The translation algebra (or Uniform Roe Algebra) of a metric space generalizes the action of a group on its space of characters and allows one to do representation theory and Fourier analysis on metric spaces. We view translations as a fundamental tool for connecting concepts.

\subsection{Entourages Approach To Scale Structures}\label{ent}
One may also define structures on set $X$ using subsets of the product space $X \times X.$ Indeed,
 closely related to scales are \textbf{entourages/controlled sets}.
\begin{Definition}
  For a space $X$, an \textbf{entourage} is a subset of $X \times X$ containing the diagonal, which is the set $\{(x,x): x \in X\}.$ A collection of entourages forms an \textbf{entourage structure}.
\end{Definition}

For a set $X$, the \textbf{diagonal} of $X$ is defined to be $\Delta = \{(x, x): x \in X\}.$ For a subset $U \subseteq X \times X$, the \textbf{inverse} of $U$ is defined to be $U^{-1} = \{(y, x) : (x, y) \in U\}.$ For two sets $U, V \subseteq X \times X,$ the \textbf{composition} (or the \textbf{product}) of $U$ and $V$ is defined as $U \circ V = \{(x,z) \mid (x,y) \in U$ and $(y,z) \in V$ for some $y \in X\}.$ For a set $E \subseteq X \times X$ and $x \in X$, let $E[x] = \{y \in X : (y,x) \in E\}.$

The above approach to scales on a group $G$ can be generalized to devise a general scheme of switching between scales on a set $X$ and entourages, i.e. subsets of $X\times X$ containing the diagonal. Namely, given an entourage $E$ in $X$, thought of as a neighborhood of the identity function on $X$, one can create its translates $g\circ E$, $g:X\to X$, and that leads to a scale on $X^X$.
In turn, that scale, when restricted to $X^\ast$ (the set of functions from a single point to $X$), gives a scale on $X$. Conversely, given a scale $\mathcal{U}$ on $X$, one can consider the entourage $\bigcup\limits_{U\in \mathcal{U}} U\times U$.

One can similarly define an order relation on entourages using composition of entourages, see section \ref{ent} or \cite{DH} for instance. That leads to the coarse structure in the sense of Roe \cite{Roe} and to emphasize that one deals with large scale point of view we use the terminology of \textbf{controlled sets} rather than entourages.
We will be doing most of our work with scales. It is already known that the entourage approach and covering approach lead to isomorphic categories, see \cite{DH}.

\begin{Definition}
A \textbf{uniform structure} on a set $X$ is a collection $\mathcal{U}$ of subsets of $X \times X$ satisfying

1) $\Delta  \subseteq U$ for all $U \in \mathcal{U}$;

2) if $U \in \mathcal{U}$, then $U^{-1} \in \mathcal{U}$;

3) for every $U \in \mathcal{U}$, there is a $V \in \mathcal{U}$ such that $V \circ V \subseteq U$;

4) if $U \in \mathcal{U}$ and $U \subseteq V$, then $V \in \mathcal{U}$;

5) if $U, V \in \mathcal{U},$ then $U \cap V \in \mathcal{U}$.

The elements of a uniform structure are called \textbf{entourages}.
 \end{Definition}

For a pseudometric space $(X,p)$, the metric uniform structure consists of all sets $E \subseteq X \times X$ which contain a set of the form $\{(a,b) : p(a, b) < r\}$ for some $r>0.$

\begin{Definition}\cite{Roe}
A \textbf{coarse structure} on a set $X$ is a collection $\mathcal{U}$ of subsets of $X \times X$ satisfying

1) $\Delta \in \mathcal{U}$;

2) if $U \in \mathcal{U}$, then $U^{-1} \in \mathcal{U}$;

3) if $U, V \in \mathcal{U},$ then $U \circ V \in \mathcal{U}$;

4) if $U \in \mathcal{U}$ and $V \subseteq U$, then $V \in \mathcal{U}$;

5) if $U, V \in \mathcal{U}$, then $U \cup V \in \mathcal{U}.$

The elements of a coarse structure are called \textbf{controlled sets}.
\end{Definition}

For an $\infty$-metric space $(X,d)$, the metric coarse structure consists of all $E \subseteq X \times X$ such that $\sup\{d(x,y): (x,y) \in E\} < M$ for some $M >0.$ The metric coarse structure is also called the \textbf{bounded coarse structure} associated to $d.$

It is known, see \cite{DH}, that coarse structures and ls-structures are equivalent concepts. 

In terms of entourages, we define a small scale entourage base as a collection $\mathcal{B}$ of symmetric entourages  such that if $E, F \in \mathcal{B}$, then there exists $G \in \mathcal{B}$ such that $G \circ G \subseteq E\cap F.$ To make a small scale entourage base into a uniform structure, add all supersets of elements of $\mathcal{B}.$
Similarly, a large scale entourage base is a collection $\mathcal{B}$ of symmetric entourages such that if $E, F \in \mathcal{B}$, then there exists $G\in \mathcal{B}$ such that $E\circ F \subseteq G$. To make a large scale entourage base into a coarse structure, add all subsets of elements of $\mathcal{B}.$

\section{Hybrid structures}
Ls-structures and ss-structures are at opposite sides of a spectrum of scale structures. To bridge the gap, we introduce intermediate structures: hybrid small scale structures and hybrid large scale structures.

\subsection{Bounded Structures}

A space with a large scale structure has a natural collection of bounded sets; namely, the bounded subsets are those which are contained in an element of some l-scale. In order to further explore connections between large and small scales, we introduce hybrid small scale spaces, which are spaces with a small scale structure along with a collection of bounded subsets. This collection cannot be an arbitrary collection of subsets, but must satisfy certain properties.

\begin{Definition}
A \textbf{bounded structure} for a set $X$ is a collection $\mathcal{B}$ of subsets of $X$ satisfying the following properties:

1) $\mathcal{B}$ contains all singleton subsets of $X$;

2) if $A\in \mathcal{B}$ and $B \subseteq A$, then $B \in \mathcal{B}$;

3) if $B_1, B_2 \in \mathcal{B}$ and $B_1 \cap B_2 \neq \emptyset,$ then $B_1 \cup B_2 \in \mathcal{B}.$\\
Condition 1) may be replaced by stating that $\mathcal{B}$ covers $X$.
\end{Definition}

Clearly, the set of  bounded sets associated to a large scale structure on a set $X$ satisfies the conditions to be a bounded structure for $X$, but it turns out that any bounded structure for a space $X$ is the set of bounded sets for some large scale structure on $X$ (\cite{ADH1} Proposition 4.1). The smallest ls-structure among them consists of all scales $\mathcal{U}$ with the property that $st(B,\mathcal{U})$ is bounded for all bounded sets $B$.

Notice that the fact of large scale structures inducing bounded structure is dual to the fact of small scale structures inducing a topology. Indeed, $B$ being bounded means that for each $x\in B$ there is an l-scale $\mathcal{U}$ with $st(x,\mathcal{U})$ containing $B$. Thus bounded sets of an ls-structure are dual to open sets of an ss-structure.

Other examples of bounded structures on a space $X$ include the collection of all subsets of $X$, the collection of all finite subsets, the collection of finite diameter subsets if $X$ is metric, and the collection of all pre-compact subsets if $X$ has a topology. A bounded structure naturally partitions a space into pieces which are, informally speaking \lq infinitely far away\rq\ - our next definition makes it more precise.

\begin{Definition}
Given a bounded structure $\mathcal{B}$ on a space $X$, define an equivalence relation $\sim$ on $X$ by $x \sim y$ if $\{x, y\} \in \mathcal{B}.$ The equivalence classes of $\sim$ are called the \textbf{$\mathcal{B}$-components} of $X$. A space is \textbf{anti-Hausdorff} if it has exactly one $\mathcal{B}$-component.
\end{Definition}

For example, in an $\infty$-metric space, points $x$ and $y$ are in the same $\mathcal{B}$-component iff $d(x, y) < \infty$ ($\mathcal{B}$ being the standard family of bounded subsets of $X$).

The concept of $X$ being anti-Hausdorff is dual to an ss-structure being Hausdorff as follows: in Hausdorff ss-structures for any two different points $x,y\in X$ one can zoom in to distinguish between $x$ and $y$ using a scale. In anti-Hausdorff spaces, given any $x,y\in X$, one can zoom out so that $x$ and $y$ belong to the same element of a scale.

\begin{Definition} Suppose $X$ has a bounded structure.
A set $B \subseteq X$ is called \textbf{weakly bounded} if for every component $C$ of $X$, the set $B \cap C$ is bounded.
\end{Definition}

Notice that the union of finitely many weakly bounded sets is again weakly bounded since the union of two bounded subsets of a component is bounded.

\begin{Definition}
A map between two spaces with bounded structures is \textbf{proper} if the inverse image of each bounded set is bounded.
\end{Definition}

\begin{Lemma}\label{WB}
If $h:X \to Y$ is a map between two spaces with bounded sets which is proper and maps bounded sets to bounded sets, then the inverse image of a weakly bounded subset of $Y$ is weakly bounded in $X$. In particular, if $X$ and $Y$ are large scale spaces, and $h: X \to Y$ is a proper, large scale continuous map, then $h^{-1}(W)$ is weakly bounded in $X$ for every weakly bounded $W \subseteq Y$.
\end{Lemma}

\begin{proof}
Let $W \subseteq Y$ be weakly bounded. Fix a component $C$ of $X$. Suppose $D_1$ and $D_2$ are components of $Y$ such that $h^{-1}(B \cap D_i) \subseteq h^{-1}(B) \cap C$ for $i = 1, 2.$ Say $x \in h^{-1}(B \cap D_1)$ and $y \in h^{-1}(B \cap D_2).$ Then $\{x, y\}$ is bounded, implying that $\{h(x), h(y)\}$ is bounded in $Y$. Thus, $D_1 \cap D_2 \neq \emptyset$, which means that $D_1 = D_2.$ Thus, $h^{-1}(B) \cap C = h^{-1}(B \cap D_1)$ is bounded by the properness of $h.$
\end{proof}

The assumption that $h$ maps bounded sets to bounded sets is needed in the above lemma. Consider, for example the identity map from $\ZZ$ to $\ZZ$, where the domain has the standard metric and the codomain has the $\infty$-metric where any two points are infinitely far apart. Then $\ZZ$ is weakly bounded in the codomain but not in the domain.

\begin{Lemma}
 Suppose $(X,\mathscr{LS})$ is a large scale space. If $B$ is weakly bounded in $X$ and $\mathcal{U}$ is
 uniformly bounded, then $st(B,\mathcal{U})$ is weakly bounded.
\end{Lemma}
\begin{proof}
 Given a component $C$ of $X$, its intersection with $st(B,\mathcal{U})$ is $st(B\cap C,\mathcal{U})$
 which is bounded.
\end{proof}

Traditionally, slowly oscillating functions are real-valued. We are proposing a vast generalization of them. 

\begin{Definition}\label{SlowlyOscillatingDef}
 A function $f:X\to Y$ from an ls-space $X$ to an ss-space $Y$ is \textbf{slowly oscillating}
 if for any two scales, $\mathcal{U}$ on $X$ and $\mathcal{V}$ on $Y$, there is a weakly bounded set $B$ in $X$
 such that $f(U)$ is contained in an element of $\mathcal{V}$ for any $U\in \mathcal{U}$ intersecting $X\setminus B$.
\end{Definition}

Slowly oscillating functions can be seen as a bridge connecting the large scale structure on the domain with the small scale structure of the codomain.

\begin{Proposition}\label{CharOfSlowlyOscillating}
  A function $f:X\to Y$ from an ls-space $X$ to an ss-space $Y$ is slowly oscillating
 if and only if for any two scales, $\mathcal{U}$ on $X$ and $\mathcal{V}$ on $Y$, there is a weakly bounded set $B$ in $X$
 such that $f(U\setminus B)$ is contained in an element of $\mathcal{V}$ for any $U\in \mathcal{U}$.
\end{Proposition}
\begin{proof}
 Notice that the condition in \ref{CharOfSlowlyOscillating} is weaker than in Definition \ref{SlowlyOscillatingDef}.
 To proceed in the other direction replace $B$ by $B':=st(B,\mathcal{U})$
 and notice that $U$ intersecting $X\setminus B'$ means $U\cap B=\emptyset$ in which case $U=U\setminus B$.
 \end{proof}

\subsection{Hybrid small scale structures}

\begin{Definition}
 A \textbf{hybrid small scale space} $X$ (an hss-space for short) is an ss-space with a bounded structure $\mathcal{B}$.
 $X$ is a \textbf{proper hss-space} if there is an s-scale $\mathcal{U}$ such that $st(B,\mathcal{U})\in \mathcal{B}$
 for all $B\in \mathcal{B}$.
 
 A function $f:X\to Y$ of hss-spaces is \textbf{hss-continuous} if it is small scale continuous and $f(B)$ is weakly bounded in $Y$ for all weakly bounded sets $B$ in $X$.
\end{Definition}

\begin{Remark}
The collection of all compact subsets of a small scale structure can always be added to the collection of bounded subsets without changing the hss-continuity of a function. This follows from the fact that the image of a compact set is compact. 
\end{Remark}

\subsection{Hybrid large scale structures}
To define a hybrid small scale space, we took a large scale concept, that of bounded sets, and attached it to a small scale space. We can also define a hybrid large scale space by dualizing this process; that is, attaching a small scale notion to a large scale space. Associated to any small scale structure on a space $X$ is a natural topology. Hence, to define a hybrid large scale space, we attach a topology to a large scale structure.

\begin{Definition}
 A \textbf{hybrid large scale space} $X$ (an hls-space for short) is an ls-space with a topology.
 $X$ is a \textbf{proper hls-space} if its bounded sets are identical with pre-compact subsets (as determined by the topology) and there is an open cover of $X$ that is uniformly bounded.
 
 A function $f:X\to Y$ of hls-spaces is \textbf{hls-continuous} if it is continuous and ls-continuous.
\end{Definition}

\section{C*-algebras and small scale structures}
Recall that $C^b(X)$ is the C*-algebra of bounded functions from a set $X$ to complex numbers $\mathbb{C}$.

\begin{Definition}
Suppose $X$ is a set.
Given a C*-subalgebra $\mathscr{C}$ of $C^b(X)$, the small scale structure on $X$ generated by $f^{-1}(\mathcal{U})$, $f\in \mathscr{C}$ and $\mathcal{U}$ a small scale on complex numbers, is denoted by $SS(\mathscr{C})$.
\end{Definition}

\begin{Remark}
Another way to describe $SS(\mathscr{C})$ is by creating the smallest topology on $X$ for which all $f\in \mathscr{C}$ are continuous, taking the smallest compactification of $X$ such that all $f\in \mathscr{C}$ extend over it, taking the unique ss-structure on that compactification, and then restricting it to $X$.
\end{Remark}

\begin{Definition}
Suppose $(X,\mathscr{SS})$ is a small scale space.
The C*-subalgebra of $C^b(X)$ consisting of all bounded and ss-continuous functions on $X$ is denoted by $C(X,\mathscr{SS})$.
\end{Definition}

The following should be regarded as a generalization of the Stone-Weierstrass Theorem (see \cite{En} or \cite{GJ}).
\begin{Theorem}\label{GenStoneWeierstrass}
If $\mathscr{C}$ is a unital C*-subalgebra of $C^b(X)$, then 
$C(X,SS(\mathscr{C}))=\mathscr{C}$.
\end{Theorem}
\begin{proof} If $f\in \mathscr{C}$, then $f$ is ss-continuous with respect to $SS(\mathscr{C})$, so $f\in C(X,SS(\mathscr{C}))$.
Conversely, if $f:X\to \mathbb{C}$ is bounded and ss-continuous with respect to $SS(\mathscr{C})$, then $f$ extends over the compactification corresponding to $\mathscr{C}$ and one can use Stone-Weierstrass Theorem (see \cite{En}) to conclude $f\in \mathscr{C}$.
Technically, Stone-Weierstrass Theorem is traditionally stated for compact Hausdorff spaces. We can reduce our situation to that case by considering the quotient function $q:X\to Y$ which identifies all points $x$ and $y$ such that $f(x)=f(y)$ for all $f\in \mathscr{C}$.
Now, the ss-structure on $Y$ induced by $\mathscr{C}$ is Hausdorff and its compactification corresponding to $\mathscr{C}$ is compact Hausdorff, so Stone-Weierstrass Theorem applies.
\end{proof}

\section{C*-algebras and bounded structures}

Every C*-subalgebra $\mathscr{C}$ of $C^b(X)$ induces a bounded structure on $X$ as follows: 

\begin{Definition}\label{BddStructureInducedByCStar}
 $B$ is \textbf{bounded} if it is contained in $C\subset X$ with the property that, for any family of non-negative functions $\{f_s\}_{s\in S}$ in $\mathscr{C}$, if
$\sum\limits_{s\in S}f_s > 0$ on $C$, then there is a finite subset $F$ of $S$ so that $\sum\limits_{s\in F}f_s > 0$ on $C$.
\end{Definition}

\begin{Proposition}
Suppose $X$ is a set and $\mathscr{C}$ is a unital C*-subalgebra of $C^b(X)$.
The bounded structure on $X$ induced by $\mathscr{C}$ is identical with all pre-compact sets in the topology on $X$ induced by $SS(\mathscr{C})$.
\end{Proposition}
\begin{proof} 
As in the proof of \ref{GenStoneWeierstrass} we can reduce the general case to that of $\mathscr{C}$ separating points of $X$.
Suppose a subset $C$ of $X$ has the property that for any family of non-negative functions $\{f_s\}_{s\in S}$ in $\mathscr{C}$ if
$\sum\limits_{s\in S}f_s > 0$ on $C$, then there is a finite subset $F$ of $S$ so that $\sum\limits_{s\in F}f_s > 0$ on $C$.
Any open family of $X$ covering $C$ can be seen as the restriction of an open family in the compactification $h(X)$ corresponding to $\mathscr{C}$ and that family can be refined by by an open family of the form $\{f_s^{-1}(0,1)\}_{s\in S}$ for some family of non-negative functions $f_s:X\to [0,1]$ belonging to $\mathscr{C}$. Since there is a finite subset $F$ of $S$ satisfying $\sum\limits_{s\in F}f_s > 0$ on $C$, the original open family of sets has a finite subfamily covering $C$.

The other direction is even simpler and is left as an exercise.
\end{proof}

\section{C*-algebras and large scale structures}

\begin{Definition}
Suppose $X$ is a set with a bounded structure $\mathscr{B}$.
Given a C*-subalgebra $\mathscr{C}$ of $C^b(X)$, a family $\mathcal{U}$ is declared to be \textbf{uniformly $(\mathscr{C},\mathscr{B})$-bounded} if the following conditions are satisfied:\\
1. $st(B,\mathcal{U})$ is weakly bounded for each weakly bounded set $B$.\\
2. for each weakly bounded set $B$, each $f\in \mathscr{C}$,
and each $\epsilon > 0$ there is a weakly bounded set $B'$ containing $B$ such that $f(U\setminus B')$ is of diameter at most $\epsilon$ for all $U\in \mathcal{U}$.
\end{Definition}

\begin{Lemma}
The family of uniformly $(\mathscr{C},\mathscr{B})$-bounded covers of $X$ forms a large scale structure on $X$ denoted by $LS(\mathscr{C},\mathscr{B})$.
\end{Lemma}
\begin{proof}  Suppose $\mathcal{U}$ and $\mathcal{V}$ are two uniformly $(\mathscr{C},\mathscr{B})$-bounded covers of $X$.
Put $\mathcal{W}=st(\mathcal{U},\mathcal{V})$. It is easy to see that $st(B,\mathcal{W})$ is weakly bounded for each weakly bounded $B$.

Suppose $B$ is weakly bounded, $f\in \mathscr{C}$,
and $\epsilon > 0$.
 Pick a weakly bounded set $B_1$ containing $B$ such that $f(U\setminus B_1)$ is of diameter at most $\epsilon/3$ for all $U\in \mathcal{U}$.
  Pick a weakly bounded set $B_2$ containing $st(B_1,\mathcal{U})$ such that $f(V\setminus B_2)$ is of diameter at most $\epsilon/3$ for all $V\in \mathcal{V}$.
Notice that $f(st(U,\mathcal{V})\setminus st(B_2,\mathcal{V}))$ is of diameter at most $\epsilon$ for all $U\in \mathcal{U}$.
\end{proof}

\begin{Example}
Suppose $(X,d)$ is a proper metric space and $C(X,d)$ is the C*-subalgebra of $C^b(X)$ consisting of ss-continuous functions on $X$ with respect to the ss-structure induced by $d$.
If $\mathscr{B}$ is the family of all pre-compact subsets of $X$, then $LS(C(X,d),\mathscr{B})$ coincides with the $C_0$-coarse structure of N.Wright (see \cite{W1} or \cite{W2}).
\end{Example}
\begin{proof} The $C_0$-coarse structure of $X$, when translated into a large scale structure, consists of pre-compact covers $\mathcal{U}$ of $X$ with the property that for any $\epsilon > 0$ there is a compact subset $K$ of $X$ all sets $U\setminus K$, $U\in \mathcal{U}$, are of diameter smaller than $\epsilon$. Therefore
$\mathcal{U}$ is a scale in $LS(C(X,d),\mathscr{B})$. Indeed, given $f\in C(X,d)$
and given $\epsilon > 0$, pick $\delta > 0$ with the property that the diameter of $f(A)$ is smaller than $\epsilon$ if the diameter of $A$ is smaller than $\delta$.
Choose a compact subset $K$ of $X$ so that the diameter of $U\setminus K$ is less than $\delta$ for all $U\in \mathcal{U}$. Hence $f(U\setminus K)$ is of diameter less than $\epsilon$ for all $U\in \mathcal{U}$.

Conversely, suppose $\mathcal{U}$ is a scale in $LS(C(X,d),\mathscr{B})$
but not a scale in the $C_0$-structure. Therefore there is
 $\epsilon > 0$ such that for each compact subset $K$ of $X$ there is $U_K\in \mathcal{U}$ with the diameter of $U_K\setminus K$ larger than $2\epsilon$.
We can pick a sequence $\{x_n\}$ in $X$ diverging to infinity such that $d(x_n,x_m) > 2\epsilon$ for all $m\ne n$ and there is a sequence $\{y_n\}$ in $X$ diverging to infinity so that
$d(x_n,y_n) > 2\epsilon$ yet $x_n,y_n\in U_n\in\mathcal{U}$ for all $n$. 
Consider the function $f$ defined as the distance from $x$ to the complement of $B(x_n,\epsilon)$ on each $B(x_n,\epsilon)$ and $f(x)=0$ for all $x$ outside the union $\bigcup\limits_{n\ge 1}B(x_n,\epsilon)$. Notice $f\in C(X,d)$
yet for any compact subset $K$ of $X$ there is $n\ge 1$ such that $x_n,y_n\in U_n\setminus K$ resulting in the diameter of $f(U_n\setminus K)$ being larger than $\epsilon$. A contradiction.
\end{proof}

\begin{Example}
Suppose $(X,d)$ is a proper metric space and $\mathscr{C}$ is either the C*-subalgebra $C_0(X)$ of $C^b(X)$ consisting of 
continuous functions on $X$ vanishing at infinity or the C*-subalgebra of $C^b(X)$ consisting of all continuous functions on $X$ that extend over the one-point compactification of $X$ as continuous functions.
If $\mathscr{B}$ is the family of all pre-compact subsets of $X$, then $LS(\mathscr{C},\mathscr{B})$ coincides with the maximal large scale structure $\mathscr{S}$ whose bounded sets are identical with pre-compact subsets of $X$.
\end{Example}
\begin{proof} Suppose $\mathcal{U}$ is a pre-compact cover of $X$ with the property that for any pre-compact subset $K$ of $X$, $st(K,\mathcal{U})$ is pre-compact. We claim
$\mathcal{U}$ is a scale in $LS(\mathscr{C},\mathscr{B})$. Indeed, given $f\in \mathscr{C}$
and given $\epsilon > 0$, pick a compact subset $K$ of $X$ so that the 
variation of $f$ on $X\setminus K$ is less that $\epsilon/2$. Hence $f(U\setminus K)$ is of diameter less than $\epsilon$ for all $U\in \mathcal{U}$.

Conversely, suppose $\mathcal{U}$ is a scale in $LS(\mathscr{C},\mathscr{B})$
but not a scale in $\mathscr{S}$. Therefore there is a pre-compact subset $K$ of $X$ such that $st(K,\mathcal{U})$ is not pre-compact.
We can pick a sequence $\{x_n\}$ in $X$ diverging to infinity and a sequence $\{y_n\}$ in $K$ converging to $y_0$ such that there are $U_n\in \mathcal{U}$ so that $x_n,y_n\in U_n\in\mathcal{U}$ for all $n$. 
Using Tietze's Extension Theorem there is a continuous function $f$ 
on the one-point compactification $\omega(X)$ of $X$ such that
$f(\omega(X)\setminus X)=0$ and $f(y_0)=1$. Notice $f\in \mathscr{C}$
yet for any compact subset $K$ of $X$ there is $n\ge 1$ such that $x_n,y_n\in U_n\setminus K$ and the diameter of $f(U_n\setminus K)$ is larger than $1/2$. A contradiction.
\end{proof}

\section{C*-algebras and hybrid small scale structures}

Notice every C*-subalgebra $\mathscr{C}$ of $C^b(X)$ induces a hybrid small scale space $HSS(\mathscr{C})$ as follows: 
the ss-structure is the smallest one making all $f\in \mathscr{C}$ ss-continuous and the bounded structure is that induced from $\mathscr{C}$ (see \ref{BddStructureInducedByCStar}).
\begin{Definition}
Suppose $(X,\mathscr{SS},\mathscr{B})$ is a hybrid small scale space.
The functor $L_0$ to hybrid large scale spaces is given by $X\to LS(C(X,\mathscr{SS}),\mathscr{B}))$.
\end{Definition}

\begin{Proposition}
 Suppose $(X,\mathscr{SS}_X,\mathscr{B}_X)$ and $(Y,\mathscr{SS}_Y,\mathscr{B}_Y)$ are hybrid small scale spaces.
 If $f:X\to Y$ is hss-continuous, then $f:L_0(X)\to L_0(Y)$ is hls-continuous.
\end{Proposition}
\begin{proof}
 Obviously, $f:X\to Y$ is continuous. Suppose $\mathcal{U}$ is a uniformly bounded family with respect to $L_0(X)$.
 To show $f(\mathcal{U})$ is uniformly bounded in $L_0(Y)$ consider an ss-continuous function $g:Y\to \mathbb{C}$ and assume $\epsilon > 0$.
 Since $g\circ f$ is ss-continuous, there is a weakly bounded subset $B$ of $X$ with the property that $diam((g\circ f)(U\setminus B)) < \epsilon$
 for all $U\in \mathcal{U}$. 
 Since $g(f(U)\setminus f(B))\subset (g\circ f)(U\setminus B)$, the diameter of $g(f(U))\setminus f(B)$ is smaller than $\epsilon$.
\end{proof}

One of the interesting cases of $L_0$ to consider is when $\mathscr{SS}$ is induced from a compactification $h(X)$ of $X$, $X$ is open in $h(X)$,
and $\mathscr{B}$ consists of all pre-compact subsets of $X$. In that case
$C(X,\mathscr{SS})$ is identical with restrictions on $X$ of continuous functions from $h(X)$ to $\mathbb{C}$.

\begin{Definition}\label{ContControlledScales}
 Suppose $X$ is a locally compact and $h(X)$ is a compactification of $X$.
 A scale $\mathcal{U}$ on $X$ is \textbf{continuously controlled by $h(X)$}
 if the following conditions are satisfied:\\
 1. $st(C,\mathcal{U})$ is pre-compact for every compact subset $C$ of $X$,\\
 2. for every neighborhood $V$ of $x\in h(X)\setminus X$
there is a neighborhood $W$ of $x$ in $V$ such that $U\cap W\ne \emptyset$ implies
$U\subset V$ for every $U\in \mathcal{U}$.
\end{Definition}

Notice the similarity of continuously controlled scales to Dugundji's covers of J.Damas \cite{MD}.

\begin{Theorem} \label{L0Char}
Suppose $X$ is a locally compact hss-space whose all bounded sets are pre-compact and its ss-structure is inherited from a compactification $h(X)$ of $X$. The family of all continuously controlled by $h(X)$ scales forms an hls-structure on $X$ that is identical with $L_0(X)$.
\end{Theorem}
\begin{proof} Suppose $\mathcal{U}$ is an l-scale of $L_0(X)$ and $V$ is an open neighborhood of $x\in h(X)\setminus X$. Pick a neighborhood $V'$ of $x$ whose closure $cl(V')$ is contained in $V$ and choose a continuous function
$f:h(X)\to [0,1]$ such that $f(cl(V'))=\{0\}$ and $f(h(X)\setminus V)=\{1\}$. There
is a compact subset $C$ of $X$ such that for any $U\in \mathcal{U}$ intersecting
$X\setminus C$, the diameter of $f(U\setminus C)$ is smaller than $0.5$.
Put $W=V'\setminus cl(st(C,\mathcal{U}))$. If $U\cap W\ne\emptyset$ for some $U\in \mathcal{U}$,
then $U\cap C=\emptyset$ and $U\cap (X\setminus V)\ne \emptyset$ is not possible. Therefore $U\subset V$
and $\mathcal{U}$ is continuously controlled by $h(X)$.

Suppose $\mathcal{U}$ is continuously controlled and $\mathcal{V}$ is an s-scale on $h(X)$.
Pick a star refinement $\mathcal{W}$ of $\mathcal{V}$. For every $x\in h(X)\setminus X$
choose a neighborhood $W_x$ such that $U\cap W_x\ne\emptyset$, $U\in \mathcal{U}$,
implies $U\subset st(x,\mathcal{W}$. Define $C$ as $h(X)\setminus \bigcup\limits_{x\in h(X)\setminus X} W_x$. Suppose $U\cap (X\setminus C)\ne\emptyset$ for some $U\in \mathcal{U}$. That means $U\cap W_x\ne\emptyset$ for some $x\in h(X)\setminus X$
resulting in $U\subset st(x,\mathcal{W}\subset V$ for some $V\in\mathcal{V}$.
That means $\mathcal{U}$ is an l-scale of $L_0(X)$.
\end{proof}

\section{C*-algebras and hybrid large scale structures}
\subsection{C*-algebras induced by hybrid large scale structures}

\begin{Definition}
Suppose $(X,\mathscr{LS},\mathscr{T})$ is a hybrid large scale space.
The C*-subalgebra of $C^b(X)$ consisting of all slowly oscillating and continuous functions on $X$ is denoted by $C(X,\mathscr{LS},\mathscr{T})$.
\end{Definition}

\begin{Proposition}
If $(\mathscr{LS},\mathscr{T})$ is a hybrid large scale structure on $X$ with bounded sets $\mathscr{B}$, then
$$\mathscr{LS}\subset LS(C(X,\mathscr{LS},\mathscr{T}),\mathscr{B}).$$
\end{Proposition}
\begin{proof}  Given $\mathcal{U}\in \mathscr{S}$, given $f\in C(X,\mathscr{S})$, and given $\epsilon > 0$ there is a weakly bounded subset $B$ of $X$ such that the diameter of $f(U\setminus B)$ is smaller than $\epsilon$ for all $U\in \mathcal{U}$. That means $\mathcal{U}\in LS(C(X,\mathscr{LS},\mathscr{T}),\mathscr{B})$.
\end{proof}

\begin{Definition}
Suppose $(\mathscr{LS},\mathscr{T})$ is a hybrid large scale structure on $X$ with bounded sets $\mathscr{B}$. $(\mathscr{LS},\mathscr{T})$ is called \textbf{reflective} if
$$LS(C(X,\mathscr{LS},\mathscr{T}),\mathscr{B})=\mathscr{LS}.$$
\end{Definition}

\begin{Theorem}
If $\mathscr{LS}$ is a metric large scale structure on $X$ with bounded sets $\mathscr{B}$ and induced topology $\mathscr{T}$, then
$$LS(C(X,\mathscr{LS},\mathscr{T}),\mathscr{B})=\mathscr{LS}.$$
\end{Theorem}
\begin{proof} 
Suppose $\mathcal{U}\in LS(C(X,\mathscr{LS},\mathscr{T}),\mathscr{B})\setminus \mathscr{S}$. 
One possibility is the existence of $U_n\in \mathcal{U}$ containing points $x_n, y_n$ with distance at least $2n$ such that both sequences $\{x_n\}$ and $\{y_n\}$ diverge to infinity. Moreover, we may assume that all balls $B(x_n,n)$ are mutually disjoint. Let $f$ be the union of functions $dist(x,X\setminus B(x_n,n))/n$. Notice $f$ is slowly oscillating and continuous, so there is a bounded subset $B$ of $X$ such that for all $n$ the diameter of $f(U_n\setminus B)$ is less than $0.5$. However, for some $n$ both points $x_n$, $y_n$ are outside of $B$ and $f(x_n)-f(y_n)=1$, a contradiction.

The other possibility is the existence of $U_n\in \mathcal{U}$ containing points $x_n, y_n$ with distance at least $2n$ such that $\{x_n\}$ diverges to infinity
and $B:=\{y_n\}$ is bounded. In that case $st(B,\mathcal{U})$ is unbounded
contradicting the facts of all elements of $LS(C(X,\mathscr{S}),\mathscr{B})$ being bounded
and $LS(C(X,\mathscr{LS},\mathscr{T}),\mathscr{B})$ containing all uniformly bounded families of $\mathscr{LS}$.
\end{proof}

\begin{Problem}
Characterize reflective hybrid large scale structures.
\end{Problem}

\begin{Proposition}\label{InclusionOfCIntoSLOfC}
If $\mathscr{C}$ is C*-subalgebra of $C^b(X)$ and $\mathscr{B}$ is a bounded structure on $X$, then
$\mathscr{C}\subset C(X,LS(\mathscr{C},\mathscr{B}))$.
\end{Proposition}
\begin{proof} Given $f\in \mathscr{C}$, given $\mathcal{U}\in LS(\mathscr{C},\mathscr{B})$, and given $\epsilon > 0$ there is $B\in \mathscr{B}$ such that for all $U\in \mathcal{U}$ the diameter of $f(U\setminus B)$ is smaller than $\epsilon$. That means $f$ is slowly oscillating and $f\in C(X,LS(\mathscr{C},\mathscr{B}))$.
\end{proof}

\begin{Question}\label{BigQuestion2}
Suppose $\mathscr{C}$ is a unital C*-subalgebra of $C^b(X)$ and $\mathscr{B}$ is a bounded structure on $X$. When does
$$C(X,LS(\mathscr{C},\mathscr{B}))=\mathscr{C}$$
hold?
\end{Question}

\begin{Definition}
Suppose $(X,\mathscr{LS},\mathscr{T})$ is a hybrid large scale space.
The functor $S_0$ to hybrid small scale spaces is given by $X\to SS(C(X,\mathscr{LS},\mathscr{T}))$, where $C(X,\mathscr{LS},\mathscr{T})$
is the C*-subalgebra of $C^b(X)$ consisting of slowly oscillating and continuous functions.
\end{Definition}

\begin{Proposition}
 Suppose $(X,\mathscr{LS}_X,\mathscr{T}_X)$ and $(Y,\mathscr{LS}_Y,\mathscr{T}_Y)$ are hybrid large scale spaces.
 If $f:X\to Y$ is hls-continuous and proper, then $f:S_0(X)\to S_0(Y)$ is hss-continuous.
\end{Proposition}
\begin{proof}
 Suppose $g:S_0(Y)\to \mathbb{C}$ is ss-continuous, i.e. $g:Y\to \mathbb{C}$ is continuous and slowly oscillating. We need to show $g\circ f:S_0(X)\to \mathbb{C}$ is ss-continuous.
 That is synonymous with $g\circ f:Y\to \mathbb{C}$ being continuous and slowly oscillating. $g\circ f$ is certainly continuous, so the only missing item
 is $g\circ f$ being slowly oscillating. Suppose $\mathcal{U}$ is a uniformly bounded family in $\mathscr{LS}_X$ and $\epsilon > 0$.
 Since $f(\mathscr{U})$ is uniformly bounded in $\mathscr{LS}_Y$ and $g$ is slowly oscillating, there is a weakly bounded subset $B$ of $Y$
 such that $diam(g(f(U))\setminus B)) < \epsilon$ for all $U\in \mathcal{U}$. 
 Put $C=f^{-1}(B)$. By \ref{WB}, $C$ is weakly bounded. Notice $(g\circ f)(U\setminus C)\subset g(f(U)\setminus B)$,
 hence the diameter of $(g\circ f)(U\setminus C)$ is smaller than $\epsilon$ for all $U\in \mathcal{U}$.
\end{proof}

\begin{Corollary}
 Suppose $X$ is a locally compact hls-space whose ls-structure consists of all scales that are continuously controlled by a compactification $h(X)$ of $X$. Then $L_0(S_0(X))=X$.
 \end{Corollary}
\begin{proof}
First, we need to show that $S_0(X)$ is $X$ with the ss-structure induced from $h(X)$.
Applying \ref{L0Char} will complete the proof.

 Suppose $f:X\to K$ is a continuous function from $X$ to a compact space $K$ and $f$ is slowly oscillating.
We can extend it over $h(X)$ by assigning to $x\in h(X)\setminus X$ the intersection of $cl(f(V\cap X))$,
$V$ ranging over all neighborhoods of $x$ in $h(X)$. That intersection consists of one point
and the extension $\tilde f:h(X)\to K$ is continuous. That means every s-scale on $S_0(X)$
is a restriction of an s-scale on $h(X)$. Notice the inclusion
$i:X\to h(X)$ is slowly oscillating and continuous. Therefore every s-scale  of $h(X)$ restricts to an s-scale of $S_0(X)$.
\end{proof}

\begin{Proposition}\label{S0L0EqualsidProp}
Suppose $X$ is a proper metric space, considered as an hss-space, then $X=S_0(L_0(X))$.
\end{Proposition}
\begin{proof}
It suffices to show that ss-continuous functions on $X$ and $S_0(L_0(X))$ are identical.

If $f:X\to \mathbb{C}$ is ss-continuous and bounded, then $f:L_0(X)\to \mathbb{C}$ is continuous, bounded, and slowly oscillating function 
by \ref{InclusionOfCIntoSLOfC} or directly from the definition of $L_0$.

Let $f:L_0(X)\to \mathbb{C}$ be a continuous, bounded, and slowly oscillating function  to $\mathbb{C}$.
Recall that $f$ is ss-continuous with respect to the metric on $X$ if, for every sequence $\{(x_n,y_n)\in X\times X\}_{n\ge 1}$ with $d(x_n,y_n)\to 0$, we have $d(f(x_n),f(y_n)) \to 0$. 

Let $\{(x_n,y_n)\in X\times X\}_{n\ge 1}$ with $d(x_n,y_n)\to 0$. 

\underline{Case 1:} The sequence of subsets $A_n = \{x_n,y_n\}$ is escaping every bounded set; i.e. for every bounded set $B$ there exists some $N\ge 1$ such that $A_n \cap B = \emptyset$ for all $n \ge N$. Because $f$ is slowly oscillating, we know there is $N \ge 1$ such that $\{\{f(x_n),f(y_n)\} : n \ge N\}$ has the diameter smaller than $\epsilon$. This implies that $d(f(x_n),f(y_n)) \to 0$.

\underline{Case 2:} $\bigcup\limits_{n}A_n$ lies in a bounded set $B$. Because $B$ is pre-compact and continuous functions are ss-continuous on pre-compact sets, we must have that $d(f(x_n),f(y_n)) \to 0$.

\underline{Case 3:} For every bounded set $B$ and $N\ge 1$, there exists $m,n \ge N$ such that $A_n \subset B$ and $A_m \cap B = \emptyset$. This means that $\{(x_n,y_n)\in X\times X\}_{n\ge 1}$ is a union of subsequences $\{(\tilde{x}_n,\tilde{y}_n)\in X\times X\}_{n\ge 1}$ such that $d(\tilde{x}_n,\tilde{y}_n)\to 0$ and $\{(\tilde{x}_n,\tilde{y}_n)\in X\times X\}_{n\ge 1}$ falls in Case 1 or in Case 2.
\end{proof}

\subsection{Hybrid large scale structures induced by C*-algebras}

\begin{Definition}
 Suppose $\mathscr{C}$ is a C*-subalgebra of $C^b(X)$. Its induced hybrid ls-space $HLS(\mathscr{C})$
 is $LS(SS(\mathscr{C}),\mathscr{B})$ with the topology induced from $SS(\mathscr{C})$ and $\mathscr{B}$ consisting of all pre-compact sets in that topology.
\end{Definition}

The following is a partial answer to Question \ref{BigQuestion2}.

\begin{Corollary}
If $X$ is a proper metric space, $\mathscr{SS}$ is its ss-structure, $\mathscr{B}$ is the bounded structure on $X$ consisting of all pre-compact sets, and $\mathscr{C}=C(X,\mathscr{SS})$, then
$$C(X,LS(\mathscr{C},\mathscr{B}))=\mathscr{C}.$$
\end{Corollary}
\begin{proof} As in the proof of \ref{S0L0EqualsidProp} one can see that every continuous and slowly oscillating function on $X$ is ss-continuous with respect to the ss-structure on $X$.
\end{proof}

\section{Noncommutative case}
Recall that a $C^*$-algebra is a norm closed and star closed sub-algebra of $B(\mathcal{H})$ (the bounded operators on the Hilbert space $\mathcal{H}$) for some Hilbert space $\mathcal{H}$. If $X$ is a set, then we will exclusively use the orthonormal basis of $\ell^2(X)$ given by the functions $\delta_x$ where $x$ ranges over $X$ (recall that $\delta_x = \chi_{\{x\}}$, the characteristic function of $\{x\}$). Every operator $a\in B(\ell^2(X))$ has a representation as an $X$ by $X$ matrix $(a_{x,y})_{x,y\in X}$ where $a_{x,y} = \langle a(\delta_x),\delta_y\rangle.$ 

The purpose of this section is to construct functors which associate scale structures to $C^*$-algebras and vice versa. There are challenges in noncommutative case that do not arise in the commutative one. For example, every commutative $C^*$-algebra $A$ is isomorphic to $C_0(X_A)$ where $X_A$ is a locally compact Hausdorff space, it is the Gel'fand spectrum on $A$, and if $A \cong C_0(Y)$ for some locally compact Hausdorff space $Y$, then it is necessarily the case that $X$ and $Y$ are homeomorphic.

\subsection{Scales, partitions of unity, and bounded operators}

A \textbf{partition of unity} on a set $X$ is a function $\phi:X\to l_1(V)$ such that $\delta_v(\phi(x))\ge 0$
for each $v\in V$  and $\sum\limits_{v\in V}\delta_v(\phi(x))=1$ for each $x\in X$. Each partition of unity on $X$ determines a scale on $X$ as follows:

\begin{Definition}
 The \textbf{support} of a partition of unity $\phi:X\to l_1(V)$ is the family of supports of functions
 $\delta_v\circ \phi$, $v\in V$, i.e. the family $\{S_v\}_{v\in V}$, where $S_v:=\{x\in X \mid \delta_v(\phi(x)) > 0\}$.
\end{Definition}

In contrast, each operator $a\in B(\ell^2(X))$ determines an entourage
$\{(x,y)\in X\times X\mid  \langle a(\delta_x),\delta_y\rangle > 0\}$.

\begin{Lemma}\label{PUsInduceGoodPUs}
 Suppose  $\phi:X\to l_1(V)$ is a partition of unity with support $\mathcal{U}$.
 There is a partition of unity $\psi:X\to l_1(X)$ whose support coarsens $\mathcal{U}$
 and refines $st(\mathcal{U},\mathcal{U})$. If $\phi$ is small scale continuous with respect to an ss-structure on $X$, then $\psi$ may be chosen small scale continuous as well.
\end{Lemma}
\begin{proof}
 Consider the subset $V_0$ of $V$  consisting of all $v\in V$ such that there is $s(v)\in X$
 satisfying $\phi(x)(s(v)) > 0$. Notice the image of $\phi$ is contained in $l_1(V_0)$. The function $s:V_0\to X$ induces a linear function $t:l_1(V_0)\to l_1(X)$ of norm $1$
 and we define $\psi$ as $t\circ \phi$.
\end{proof}

Notice the connection between partitions of unity and and bounded operators in $B(l_2(X))$:
\begin{Lemma}\label{PUsInduceBoundedOperators}
 Every partition of unity 
 $\phi:X\to l_1(X)$ induces a bounded operator $l_2(X)\to l_2(X)$ whose restriction to $X$ equals $\phi$.
\end{Lemma}

\subsection{Large scale structures induced by C*-algebras}
Roe \cite{Roe} defined the uniform C*-algebra $C^*_u(X)$ in case of metric spaces $X$ as the one generated by
bounded operators whose support is a controlled set in the coarse structure induced by the metric of $X$.
Obviously, the same definition works for any coarse space $X$.

A reverse process is to define
large scale structures on a set $X$ given a C*-subalgebra $\mathcal{A}\subset  B(\ell^2(X))$.
We will do it in a manner analogous to the construction of word metrics on finitely generated groups.

\begin{Definition}
A family of bounded operators $\mathcal{F}\subset  B(\ell^2(X))$ is *-symmetric if $a^*\in \mathcal{F}$
for each $a\in \mathcal{F}$.
\end{Definition}

\begin{Definition}
 Given a *-symmetric family of bounded operators $\mathcal{F}\subset  B(\ell^2(X))$ and a cover $\mathcal{U}$ of $X$, $\mathcal{U}$ is called $\mathcal{F}$-\textbf{bounded} if there is $n\ge 1$ such that for all $x\ne y$ belonging to some $U\in \mathcal{U}$ there is a chain $x_1=x, \ldots, x_n=y$
 of points in $U$ so that for each $i < n$ either $x_i=x_{i+1}$ or 
 $|\langle a(\delta_{x_i}),\delta_{x_{i+1}}\rangle| \ge 1$ for some $a\in\mathcal{F}$.
\end{Definition}

\begin{Remark}
 One can define a pseudo-metric here.
\end{Remark}

\begin{Definition}
 Given a unital *-subalgebra $\mathcal{A}\subset  B(\ell^2(X))$ define the ls-structure $\mathscr{LS}(\mathcal{A})$ on $X$ by considering all $\mathcal{F}$-\textbf{bounded} covers of $X$, where $\mathcal{F}$ is a finite *-symmetric subset of $\mathcal{A}$. 
 \end{Definition}

\begin{Proposition}\label{ClosureGivesTheSameLS}
A unital *-subalgebra $\mathcal{A}\subset  B(\ell^2(X))$ induces the same large scale structure on $X$ as its closure.
\end{Proposition}
\begin{proof}
 Given a finite *-symmetric subset $\mathcal{F}$ of $cl(\mathcal{A})$ choose
 a finite *-symmetric subset $\mathcal{G}$ of $\mathcal{A}$ such that for any $a \in \mathcal{F}$
 there is $b\in \mathcal{G}$ satisfying $|a-b| < 0.5$. Notice that any $0.5\cdot \mathcal{G}$-bounded cover of $X$
 is $\mathcal{F}$-bounded.
\end{proof}

\begin{Proposition}\label{C*IsSmaller}
 If $(X,\mathscr{LS})$ is a large scale space, then the ls-structure induced by the uniform Roe algebra $C^*_u(X)$ is smaller than or equal the original ls-structure of $X$.
 \end{Proposition}
\begin{proof}
Suppose $\mathcal{U}$ is $\mathcal{F}$-\textbf{bounded} for some finite *-symmetric subset  $\mathcal{F}$ of $C^*_u(X)$ with respect to an integer $n > 1$. In view of \ref{ClosureGivesTheSameLS} we may assume supports of elements of $\mathcal{F}$ are controlled sets
in $(X,\mathscr{LS})$.
There is a uniformly bounded cover $\mathcal{V}$ of $X$ such that $\{x,y\}$ not contained in any element of $\mathcal{V}$
 implies $|\langle a(\delta_{x}),\delta_{y}\rangle| =0$ for all $a\in \mathcal{F}$.
That implies easily that $\mathcal{U}$ is contained in the $n$-th star of $\mathcal{V}$, hence is uniformly bounded in the original structure of $X$.
\end{proof}

\begin{Proposition}\label{C*ContainsCoversOfFiniteMultiplicity}
 If $(X,\mathscr{LS})$ is a large scale space  then the ls-space induced by the uniform Roe algebra $C^*_u(X)$ contains all uniformly bounded covers in the original ls-structure of $X$ that have finite multiplicity.
 \end{Proposition}
\begin{proof}
Suppose $\mathcal{U}$ is a uniformly bounded cover in $X$ of multiplicity $m$. For each $U\in \mathcal{U}$ pick $x(U)\in U$. Let $T:\ell^2(X) \to \ell^2(X)$ be given by 
$T(\delta_y) = \sum\limits_{y\in U\in \mathcal{U}} \delta_{x(U)}$. Clearly $T$ is a bounded operator. Notice that $\mathcal{U}$ is $\{T,T^*\}$-bounded
for $n=2$. 
\end{proof}

The following question is related to a question from \cite{WinZach} where it is shown that
the nuclear dimension of the uniform Roe algebra does not exceed the asymptotic dimension of $X$ (see \cite{GuWiYu} for an alternative proof).

\begin{Question}\label{NuclearDimVsAsdim}
 Suppose $\mathcal{A}\subset  B(\ell^2(X))$ is a unital $C^*$-subalgebra of finite nuclear dimension.
 Is the asymptotic dimension of the induced large scale structure on $X$ finite?
\end{Question}

Notice H.Sako \cite{Sako} proved that nuclearity of the uniform Roe algebra is equivalent to $X$ having Property A in case
of $X$ being uniformly locally finite. Therefore one may ponder a similar question to \ref{NuclearDimVsAsdim} for Property A.

\begin{Theorem}
 If $(X,\mathscr{LS})$ is a coarsely finitistic large scale space, then the ls-space induced by the uniform Roe algebra $C^*_u(X)$ yields the original ls-structure of $X$.
 \end{Theorem}
\begin{proof}
$(X,\mathscr{LS})$ being coarsely finitistic (see \cite{CDV2})  means that each uniformly bounded cover of $X$ can be coarsened to a uniformly bounded cover of $X$ that is of finite multiplicity.
Apply \ref{C*ContainsCoversOfFiniteMultiplicity} and \ref{C*IsSmaller}.
\end{proof}

\begin{Corollary}
 If $(X,\mathscr{LS})$ is a large scale space of bounded geometry or of finite asymptotic dimension, then the ls-space induced by the uniform Roe algebra $C^*_u(X)$ yields the original ls-structure of $X$.
 \end{Corollary}
 
 \subsection{Small scale structures induced by C*-algebras}
 
 \begin{Definition}
 Given a *-subalgebra $\mathcal{A}\subset  B(\ell^2(X))$ define the small scale structure $\mathscr{SS}(\mathcal{A})$ on $X$ 
 as the smallest ss-structure for which all functions $X\to l_2(X)$ induced by $a\in \mathcal{A}$ are ss-continuous.
  \end{Definition}

 \begin{Proposition}\label{ClosureGivesTheSameSS}
A *-subalgebra $\mathcal{A}\subset  B(\ell^2(X))$ induces the same small scale structure on $X$ as its closure.
\end{Proposition}
\begin{proof}
 Suppose $\epsilon > 0$, $a\in B(\ell^2(X))$, and $\mathcal{U}$ is the scale on $X$ obtained via $a$
 by pulling back all $\epsilon$-balls in $l_2(X)$. Approximate $a$ by $b$ so that $|a-b| < \epsilon /4$
 and notice that the scale on $X$ obtained via $b$
 by pulling back all $\epsilon/4$-balls in $l_2(X)$ refines $\mathcal{U}$.
\end{proof}

\begin{Definition}
 Given a small scale structure $\mathscr{SS}$ on a set $X$ consider all
 bounded operators $a:l_2(X)\to l_2(X)$ for which the function $X\to l_2(X)$ is ss-continuous.
 They form a C*-subalgebra of $B(l_2(X))$ which we denote by $C^\ast(\mathscr{SS})$.
  \end{Definition}

\begin{Problem}\label{SSvC*Problem}
 Characterize small scale structures $\mathscr{SS}$ on $X$ which are equal to the one induced by
 $C^\ast(\mathscr{SS})$.
\end{Problem}

To give a partial answer to \ref{SSvC*Problem}, we introduce a small scale analog of paracompactness.
One such analog was defined in \cite{CDV2} under the name of small scale paracompactness, so to avoid a conflict our
new notion is named differently in spite that it may be actually the best analog of paracompatness.
\begin{Definition}
 A small scale structure $\mathscr{SS}$ is \textbf{ss-strongly paracompact} if for each 
 scale $\mathcal{U}$ in $\mathscr{SS}$ there is an ss-continuous partition of unity
 $\phi:X\to \Delta(S)$ whose support is a scale in $\mathscr{SS}$ smaller than $\mathcal{U}$.
\end{Definition}

A result of M. Zahradnik \cite{Zah} implies that the Hilbert space is not ss-strongly paracompact. A simpler proof
(connected to the fact that the Hilbert space does not have Property A of G.Yu) can be found in \cite{CDV2}).

\begin{Theorem}
 If a small scale structure $\mathscr{SS}$ on $X$ is ss-strongly paracompact,  then it is equal to the one induced by
 $C^\ast(\mathscr{SS})$.
\end{Theorem}
\begin{proof}
 Given a small scale $\mathcal{U}$ in $\mathscr{SS}$, pick an ss-continuous partition of unity  $\phi:X\to l_1(X)$ whose support is a scale in $\mathscr{SS}$ smaller than $\mathcal{U}$.
 Notice $\phi$ induces an element of $C^\ast(\mathscr{SS})$, so $\mathcal{U}$ belongs to the ss-structure induced by
 $C^\ast(\mathscr{SS})$ (use \ref{PUsInduceGoodPUs}).
\end{proof}

\end{document}